\documentclass[12pt]{amsart}

\usepackage{amsmath}
\usepackage{amssymb}
\usepackage{amsthm}
\usepackage{graphicx}
\usepackage{amscd}
\usepackage{url}

\newtheorem{theorem}{\sc Theorem}[section]
 
 \newtheorem{corollary}[theorem]{\sc Corollary}
 \newtheorem{proposition}[theorem]{\sc Proposition}
 
\newtheorem{example}[theorem]{\sc Example}
\newtheorem{remark}[theorem]{\sc Remark}

\newcommand{\F}{\mathbb{F}}
\newcommand{\PG}{\mathrm{PG}}

\newcommand{\B}{\mathcal{B}}
\newcommand{\D}{\mathcal{D}}
\newcommand{\SSS}{\mathbb{S}}
\newcommand{\wt}{\mathrm{wt}}

\begin{document}

\title{Scattered Spaces in Galois Geometry}

\author{Michel Lavrauw}
\address{Michel Lavrauw\\ Universit\`a degli Studi di Padova, Italy, {\url{michel.lavrauw@unipd.it}}}


\date{January 27, 2016}

\maketitle

\begin{abstract}
This is a survey paper on the theory of scattered spaces in Galois geometry and its applications.
\end{abstract}

\section{Introduction and motivation} \label{lavrauw:sec:introduction}

Given a set $\Omega$ and a set $S$ of subsets of $\Omega$, a subset $U\subset \Omega$ is called {\em scattered} with respect to $S$ if $U$ intersects each element of $S$ in at most one element of $\Omega$. In the context of Galois Geometry this concept was first studied in 2000 \cite{BaBlLa2000}, where the set $\Omega$ was the set of points of the projective space $\PG(11,q)$ and $S$ was a $3$-spread of $\PG(11,q)$. The terminology of a {\em scattered space}\footnote{This notion of scattered spaces is not to be confused with scattered spaces in Topology,
where a space $A$ is called {\em scattered} if every non-empty subset $T\subset A$ contains a point isolated in $T$.} was introduced later in \cite{BlLa2000}.  
The paper \cite{BaBlLa2000} was motivated by the theory of blocking sets, and it was shown that there exists a $5$-dimensional subspace, whose set of points $U \subset \Omega$ is scattered with respect to $S$, which then led to an interesting construction of a $(q+1)$-fold blocking set in $\PG(2,q^4)$. The notion of ``being scattered" has turned out to be a useful concept in Galois Geometry. 
This paper is motivated by the recent developments in Galois Geometry involving scattered spaces. The first part aims to give an overview of the known results on scattered spaces (mainly from \cite{BaBlLa2000}, \cite{BlLa2000}, and \cite{Lavrauw2001}) and the second part gives a survey of the applications.

\section{Notation and terminology}

A {\em $t$-spread} of a vector space $V$ is a partition of $V\setminus \{0\}$ by subspaces of constant dimension $t$. Equivalently, a $(t-1)$-spread of $\PG(V)$ (the projective space associated to $V$) is a set of $(t-1)$-dimensional subspaces partitioning the set of points of $\PG(V)$. Sometimes the shorter term {\em spread} is used, when the dimension is irrelevant or clear from the context. Two spreads $S_1$, $S_2$ of $\PG(V)$ (respectively $V$) are called {\em equivalent} if there exists a collineation $\alpha$ of $\PG(V)$ (respectively, an element $\alpha$ of ${\mathrm{\Gamma L(V)}}$)  such that $S_1^\alpha = S_2$.

A standard construction of a spread (going back to Segre \cite{Segre1964}) is the following. 
We sketch the construction in the context of the larger framework of {\em field reduction} techniques on which 
we will elaborate in Section \ref{subsec:linear_sets}.
Consider any $\F_q$-vector space isomorphism from $\F_{q^t}\rightarrow \F_q^t$, and extend this to an isomorphism $\varphi$ between the $\F_q$-vector spaces $\F_{q^t}^r$ and $\F_q^{rt}$. For each nonzero vector $v\in V(r,q^t)$ consider the vector space $S_v=\{\varphi(\lambda v)~:~\lambda \in \F_{q^t}\}$. One easily verifies that the set ${\mathcal{D}}_{r,t,q}:=\{S_v~:~ v \in V(r,q^t)\}$ defines a $t$-spread in $V(rt,q)$.
A spread $S$ is called {\em Desarguesian} if $S$ is equivalent to $\D_{r,t,q}$ for some $r$, $t$ and $q$.

Let $D$ be any set of subspaces in $V(n,q)$. A subspace $W$ of $V$ is called {\em scattered with respect to $D$} if $W$ intersects each element of $D$ in at most a $1$-dimensional subspace. Equivalently, a subspace of a projective space $\PG(V)$ is called scattered with respect to a set of subspaces $D$ of $\PG(V)$ if it intersects each element of $D$ in at most a point. In this paper $D$ will typically be a spread.
We note that we will often switch between vector spaces and projective spaces, assuming that the reader is familiar with both terminologies. To avoid overcomplicating the notation, we will use the same symbol (for instance $D$) for subset of the subspaces of a projective space and its associated object in the underlying vector space.
We will make sure that there is no ambiguity concerning vector space dimension and projective dimension. 

If $D$ is any set of subspaces of a vector space or a projective space, and $W$ is a subspace of the same space, then by $\B_D(W)$ we denote the set of elements of $D$ which have a nontrivial intersection with $W$. If there is no confusion possible, then we also use the simplified notation $\B(W)$.

\section{Scattered spaces}\label{lavrauw_section:basics}
We start this section with a number of examples illustrating some of the difficulties that arise in the study of scattered spaces with respect to spreads.
\begin{example}\label{lavrauw_example:easy}
\begin{enumerate}
\item If $D$ is a spread of lines in $\PG(3,q)$ then every line not contained in the spread is scattered w.r.t. $D$. Also, since $|D|=q^2+1$ and a plane contains $q^2+q+1$ points, no plane of $\PG(3,q)$ is scattered w.r.t. $D$.
\item If $D$ is a spread of planes in $\PG(5,q)$ then every line not contained in an element of the spread is scattered w.r.t. $D$. Also since $|D|=q^3+1$ and a solid (a 3-dimensional projective space) contains $q^3+q^2+q+1$ points, no solid of $\PG(3,q)$ is scattered w.r.t. $D$. The existence of a scattered plane is not immediately clear, but this will follow from one of the results in the next sections (Theorem \ref{general lower bound}).
\item If $D$ is a spread of lines in $\PG(5,q)$ then it is easy to see that the dimension of a scattered subspace of $\PG(5,q)$ w.r.t. $D$ cannot exceed 3 (a solid). Also any such spread allows a scattered line (trivial), and a scattered plane (Theorem \ref{general lower bound}). On the other hand, the existence of a scattered solid depends on the spread. We will see that a Desarguesian spread does not allow scattered solids (by Theorem \ref{general upper bound}), but there are spreads which do (by Theorem \ref{scattering spread}).
\end{enumerate}
\end{example}

\subsection{Maximally vs maximum scattered}\label{lavrauw_section:max_vs_max}
It is important to observe the distinction between the following two definitions. A subspace $U$ is called {\em maximally scattered} w.r.t. a spread $D$ if $U$ is not contained in a larger scattered space.
A subspace $U$ is called {\em maximum scattered} w.r.t. a spread $D$ if any scattered space $T$ w.r.t. $D$ satisfies $\dim T \leq \dim U$. As we will see in the following example, there exist maximally scattered spaces which are not maximum scattered.
\begin{example}
Consider the irreducible polynomial $f(x)=x^6+x^4+x^3+x+1\in \F_2[x]$, put $\F_{2^6}=\F_2[x]/(f(x))$ and consider 
the set of subspaces 
$D=\{S_{u,v}~:~ u,v \in \F_2^6\}$, where $S_{u,v}=\langle (M^k u, M^k v)~:~k\in \{0,1,\ldots, 5\}\rangle$ of $V(12,2)$, where $M$ is the companion matrix of $f(x)$. This is the standard construction (by field reduction see Section \ref{subsec:linear_sets}) of a Desarguesian spread, in this case a Desarguesian $6$-spread in $V(12,2)$. The subspace $U_5$ spanned by the rows of the matrix
\begin{displaymath}
\left [
\begin{array}{cccccccccccc}
 1 & 0 & 0 & 0 & 0 & 0 & 0 & 0 & 0 &1 &1 & 0 \\
 0 &1 & 0 & 0 & 0 & 0 & 0 & 0 &1 & 0 & 0 & 0 \\
  0 & 0 &1 & 0 & 0 & 0 & 0 & 0 & 0 & 0 &1 &1 \\
  0 & 0 & 0 &1 & 0 &1 & 0 & 0 & 0 & 0 & 0 & 0 \\
  0 & 0 & 0 & 0 &1 &1 &1 &1 &1 &1 &1 & 0 \\
\end{array}
\right ]
\end{displaymath}
is a maximally scattered 5-dimensional subspace of $V(12,2)$. If $\varphi: \F_{2^6}\rightarrow \F_2^6$ is any vector space isomorphism, then the subspace $U_6:=\{(\varphi(\alpha),\varphi(\alpha^2))~:~ \alpha \in \F_2^6\}$ is maximum scattered. It follows that $U_5$ is maximally scattered but not maximum scattered.
This example was constructed using the GAP-package FinInG (see \cite{gap}, \cite{fining}). 
\end{example}

\subsection{A lower bound on the dimension of maximally scattered spaces}\label{lavrauw_section:lower}
It is obvious that every line of a projective space, which is not contained in an element of a spread $D$, is scattered with respect to $D$, so a maximally scattered space has vector dimension at least $2$. The following theorem gives a lower bound on the dimension of a maximally scattered space in terms of the dimension of the space and the dimension of the spread elements. Its proof (see \cite{BlLa2000}) is purely combinatorial and gives a method to extend a scattered space in case the bound is not attained.
\begin{theorem}\label{general lower bound}{\rm \cite[Theorem 2.1]{BlLa2000}}\\
If $U$ is a maximally scattered subspace w.r.t. a $t$-spread in $V(rt,q)$, then $\dim U \geq \lceil (rt-t)/2\rceil +1$.
\end{theorem}
Theorem \ref{general lower bound} implies the existence of a scattered plane w.r.t. a plane spread in $\PG(5,q)$, answering one of the questions from Example \ref{lavrauw_example:easy}.

\subsection{An upper bound on the dimension of scattered spaces}\label{lavrauw_section:upper}

Let $S$ be a $(t-1)$-spread in $\mathrm{PG}(rt-1,q)$. The number of
spread elements is $(q^{rt}-1)/(q^t -1)$ $ = q^{(r-1)t} +
q^{(r-2)t}+ \ldots + q^t +1$. Since a scattered subspace can
contain at most one point of every spread element, the number of
points in a scattered space must be less than or equal to the
number of spread elements. This gives the following trivial upper
bound.
\begin{theorem}\label{general upper bound}{\rm \cite[Theorem 3.1]{BlLa2000}}\\
If $U$ is scattered w.r.t. a $t$-spread in $V(rt,q)$, then $\dim U \leq rt-t$.
\end{theorem}

Note that for a line spread in $\mathrm{PG}(3,q)$ the upper and lower
bound coincide. But this is quite exceptional. In fact, excluding trivial cases 
we may assume that $t$ and $r$ are both at least 2, and it
follows that the exact dimension of a maximum
scattered space is determined by the lower and upper bounds only for $(r,t) \in \{ (2,2),(2,3)\}$, i.e., a line
spread in $\mathrm{PG}(3,q)$ and a plane spread in $\mathrm{PG}(5,q)$. 
The projective dimension of a maximum scattered space in these cases is respectively 1, the
dimension of a line, and 2, the dimension of a plane. There is a
large variety of spreads and there is not much one can say
about the possible dimension of a scattered subspace with respect
to an arbitrary spread (see Section \ref{subsec:scattering spreads}). This is one of the reasons to consider scattered
spaces with respect to a Desarguesian spread (see Section \ref{subsec:desarguesian spreads}). Another reason is
the correspondence between the elements of a Desarguesian spread
and the points of a projective space over an extension field (so-called {\em field reduction}, see Section \ref{subsec:linear_sets}). 
However before we proceed, in the following section we show that the upper bound from 
Theorem \ref{general upper bound} cannot be improved without restrictions on the spread.

\subsection{Scattering spreads with respect to a subspace}\label{subsec:scattering spreads}
A spread $D$ is called {\it a scattering spread with respect to a subspace} $U$, if this subspace $U$ is scattered
with respect to the spread $D$.

\begin{theorem}\label{scattering spread}{\rm \cite[Theorem
3.2]{BlLa2000}}\\
If $W$ is an $(rt-t-1)$-dimensional subspace of $\mathrm{PG}(rt-1,q)$, $r\geq 2$,
then there exists a  scattering $(t-1)$-spread $\mathcal S$ with respect to $W$.
\end{theorem}
This theorem shows that there is no room for improvement of Theorem \ref{general upper bound} without assuming some extra properties on the spread.
Up to now, the only spreads that have been investigated in detail are the Desarguesian spreads (see Section \ref{subsec:desarguesian spreads}).

\subsection{Scattered spaces w.r.t. Desarguesian spreads}\label{subsec:desarguesian spreads}
Let $S$ be a $(t-1)$-spread of $\PG(rt-1,q)$ and consider the following incidence structure. First embed $\PG(rt-1,q)$ as a hyperplane $H$ in $\Sigma=\PG(rt,q)$. Denote the set of points of $\Sigma \setminus H$ by $\mathcal P$ and
the set of $t$-dimensional subspaces of $\Sigma$ which intersect $H$ in an element of $S$ by $\mathcal L$. Define an incidence relation $\mathcal I$ on $({\mathcal{P}}\times {\mathcal{L}}) \cup ({\mathcal{L}}\times {\mathcal{P}})$ by symmetric containment. Then the incidence structure ${\mathcal{D}}(S)=({\mathcal{P}},{\mathcal{L}},{\mathcal{I}})$ is a design with parallelism, also called {\em Sperner space} or {\em S-space} (see e.g. \cite{BaCo1974}). More precisely,
${\mathcal{D}}(S)$ is a $2-(q^{rt},q^t,1)$ design such that for each anti-flag $(x,U)$, there exists exactly one element of $\mathcal L$ which is incident with $x$ and parallel to $U$ (two elements of $\mathcal L$ are called {\em parallel} if their intersection is an element of $S$).
Moreover, the design ${\mathcal{D}}(S)$ is a Desarguesian affine space if and only if $S$ is a Desarguesian spread (see \cite[Theorem 2]{BaCo1974}). This correspondence is crucial in the proof of  the following theorem 
and is of central importance for many of the applications of scattered spaces.
We remind the reader that ${\mathcal{D}}_{r,t,q}$ denotes the Desarguesian $t$-spread in $V(rt,q)$.
\begin{theorem}\label{Desarguesian upper bound}
{\rm \cite[Theorem 4.3]{BlLa2000}}\\
If $U$ is a scattered subspace w.r.t. ${\mathcal{D}}_{r,t,q}$, then
$\dim U \leq rt/2$.
\end{theorem}
It was also shown in \cite{BlLa2000} that this upper bound is tight whenever $r$ is even.
\begin{theorem}{\rm \cite{BlLa2000}}
If $r$ is even, then there exists a scattered subspace w.r.t. ${\mathcal{D}}_{r,t,q}$ in $V(rt,q)$ of dimension $rt/2$.
\end{theorem}
For $r$ odd, the exact dimension of a maximum scattered space is in general not known.
The following theorem gives a lower bound on the dimension of a maximum scattered space.
\begin{theorem}\label{thm:scattered_existence}{\rm\cite{BlLa2000}}
The dimension of a maximum scattered subspace w.r.t. ${\mathcal{D}}_{r,t,q}$ in $V(rt,q)$ is at least $r'k$ where
$r' | r$, $(r',t)=1$, and $r'k$ is maximal such that
$$
k<(rt-t+3)/2~\mbox{for $q=2$ and $r'=1$}
$$
and 
$$
r'k<(rt-t+r'+3)/2~\mbox{otherwise.}
$$
\end{theorem}
We conclude this section with an overview of the values of $r,t,q$, with $r$ odd and $t$ even, for which maximum scattered spaces w.r.t. $\D_{r,t,q}$ have been constructed.
Note that for $t=2$, the existence of a scattered r-space w.r.t. $\D_{r,2,q}$ easily follows by considering an appropriate maximal subspace lying on the Segre variety $S_{r,2}(q)$, or equivalently a subspace $\F_q^r\otimes v$ ~of ~$\F_q^r\otimes \F_q^2$, for some nonzero vector $v\in \F_q^2$ (see e.g. Section 1.6 and in particular Theorem 1.6.4 of \cite{Lavrauw2001} for more details).
For $t=4$ and $r=3$, a 6-dimensional scattered space w.r.t. $\D_{3,4,q}$ was constructed in \cite{BaBlLa2000} (see more on this in Section \ref{subsec:blocking_sets}). A much more general result (including the construction from \cite{BaBlLa2000}) was recently obtained by Bartoli et al. in \cite{BaGiMaPoPrep}. They constructed scattered linear sets of rank $rt/2$ (see Section \ref{subsec:linear_sets} below for definitions) in $\PG(r-1,q^t)$ for many parameters $r$, $t$ and $q$. As a corollary one obtains the existence of maximum scattered spaces w.r.t. $\D_{r,t,q}$ in the following cases.

\begin{corollary} {\rm (From \cite{BaGiMaPoPrep})}
There exist scattered spaces w.r.t. ${\mathcal{D}}_{r,t,q}$, $t$ even, of dimension $rt/2$ in the following cases: (i) $q=2$ and $t\geq 4$; (ii) $q\geq 2$ and $t\not \equiv 0$ mod $3$; (iii) $q\equiv 1$ mod $3$ and $t \equiv 0$ mod $3$.
\end{corollary}
Apart from the computational examples from \cite{BaGiMaPoPrep}, for $t=6$ and $q\in \{3,4,5\}$, the existence of scattered spaces of dimension $rt/2$ w.r.t. $\D_{r,t,q}$, with $r$ odd and $t$ even, remains open for $t\equiv 0 \mod 3$, $q \not \equiv 1 \mod 3$, and $q>2$.

\section{Applications}
\subsection{Translation hyperovals}
We start the section with one of the earliest applications of scattered spaces: hyperovals of translation planes. 
We assume the reader is familiar with the notion of a projective plane of order $q$. All necessary background (and much more) can be found in for instance \cite{Dembowski1968} or \cite{HuPi1973}. 
A {\em hyperoval} in a projective plane of order $q$ is a set $\mathcal H$ of $q+2$ points, no three of which are collinear, i.e. no line of the plane contains three points of $\mathcal H$. The existence of a hyperoval in a projective plane $\pi$ implies the order $q$ of $\pi$ to be even.

Let $\pi$ be a projective plane. A {\em perspectivity} of $\pi$ is a collineation $\alpha$ for which there exists a point-line pair $(x, \ell)$ such that  $\alpha$ fixes each line on $x$ and each point on $\ell$, and $\alpha$ is then also called an {\em $(x,\ell)$-perspectivity}. If $x$ is on $\ell$ then $\alpha$ is called an {\em $(x,\ell)$-elation}, otherwise $\alpha$ is called an {\em $(x,\ell)$-homology}. The point $x$ is the {\em center} and the line $\ell$ is the {\em axis} of $\alpha$.
A {\em translation plane} is a projective plane which contains a line $\ell_\infty$, such that for each point
$x$ on $\ell_\infty$ the group of elations with center $x$ and axis $\ell_\infty$ acts transitively
on the points of $m\setminus \{x\}$ for each line $m$ on $x$.
Equivalently, the automorphism group $G$ of the plane is called {\em $(x,\ell_\infty)$-transitive} for each
point $x$ on $\ell_\infty$.

A hyperoval $\mathcal H$ is called a 
{\em translation hyperoval} if there exists a group $G$ of $q$ elations with common axis $\ell$, fixing $\mathcal H$.
The line $\ell$ is a 2-secant, called the {\em translation line} of $\mathcal H$, and the group $G$ acts transitively on the points of ${\mathcal{H}}\setminus \ell$. Translation hyperovals in the Desarguesian projective plane $\PG(2,q)$ were classified by Payne in 1971 \cite{Payne1971}. For non-Desarguesian projective planes, and in particular for translation planes, the classification remains open. Denniston \cite{Dennistion1979} and Korchm\'aros \cite{Korchmaros1986} constructed translation ovals in non-Desarguesian translation planes, while Jha and Johnson \cite{JhJo1992} proved that for each non-prime integer $N>3$, there exists a non-Desarguesian translation plane of order $2^N$ admitting a translation hyperoval. Note that not every projective plane admits translation hyperovals since some planes (e.g. Figueroa planes) don't have enough translations.

The study of translation planes is equivalent to the study of spreads of a vector space whose dimension is twice
the dimension of the elements of the spread. This correspondence is due to the Andr\'e-Bruck-Bose (ABB) construction of a translation plane from a spread and vice versa. In fact, the ABB-construction is a special case of the construction of the design with parallelism $\D(S)$ mentioned above (with $r=2$). In this case, $\D(S)$ is a $2-(q^{2t},q^t,1)$ design with parallelism, i.e. an affine plane. The corresponding projective plane is denoted by $\pi(S)$. The equivalence between spreads and translation planes was obtained in 1954 by Andr\'e \cite{Andre1954}, using a group-theoretic point of view. The geometric construction as presented above was published by Bruck and Bose \cite{BrBo1964} in 1964. The following theorem shows the equivalence between translation hyperovals in translation planes (sharing the same axis) and scattered spaces, and is comparable to \cite[Theorem 5]{JhJo1992}, although the terminology in \cite{JhJo1992} is quite different.

\begin{theorem}\label{thm:scattered-hyperoval}
A translation plane $\pi(S)$ of order $2^{t}$ with translation line $\ell_\infty$ contains a translation hyperoval with translation line $\ell_\infty$ if and only if the spread $S$ in $V(2t,2)$ admits a scattered space of dimension $t$.
\end{theorem}
\begin{proof}
We prove the first part of the theorem in a projective setting. Let $S$ be a $(t-1)$-spread in $\Sigma=\PG(2t-1,2)$, and $U$ a scattered $(t-1)$-space w.r.t $S$. Embed $\Sigma$ as a hyperplane in $\Sigma^*$ and choose a $t$-dimensional
subspace $K_U$ of $\Sigma^*$ with $K_U\cap \Sigma=U$. Let $\mathcal H$ denote the 
set of points in $K_U\setminus U$ together with the two points (call them $x$ and $y$) at infinity corresponding to the spread elements which are disjoint from $U$. We claim that $\mathcal H$ is a hyperoval of $\pi(S)$. 
To prove this claim, consider a line $\ell$ in $\pi(S)$.
If $\ell$ contains $x$ (respectively $y$) then the $t$-space of $\Sigma^*$ corresponding to $\ell$ intersects $K_U$ in exactly one point $z$. In this case, the line $\ell$ contains exactly two points of $\mathcal H$, namely $z$ and $x$ (respectively $y$).
If $\ell$ does not contain $x$ or $y$, then the $t$-space $\bar{\ell}$ of $\Sigma^*$ corresponding to $\ell$ intersects $\Sigma$ in an element $R\in S$ which intersects $U$ in a point $u$. There are two possibilities: either $\bar{\ell}$ intersects $K_U$ only in the point $u$, or $\bar{\ell}$ intersects $K_U$ in a line $\{u,u',u''\}$. In the first case, $\ell$ is external to $\mathcal H$. In the latter case, the line $\ell$ intersects $\mathcal H$ in exactly two points $u'$ and $u''$. This shows that no three points of $\mathcal H$ are collinear in $\pi(S)$ (note that the line $\ell_\infty$ meets $\mathcal H$ in the two points $x$ and $y$). It follows that $\mathcal H$ is a set of $2^t+2$ points, no three of which are collinear, i.e. $\mathcal H$ is a hyperoval. The existence of the group $H$ of $q$ elations, with the line at infinity as the common axis, immediately follows from the construction, as $H$ corresponds to the translation group 
stabilising $K_U\setminus U$.

Conversely, suppose that $\pi(S)$ contains a translation hyperoval $\mathcal H$. Let $T$ denote the translation group of $\pi(S)$ and $H\leq T$ the translation group of $\mathcal H$. If $T_z\leq T$ denotes the group of $(z,\ell_\infty)$-elations, then the spread $S$ coincides with the set $\{T_z~:~z \in \ell_\infty\}$. All groups $T$, $H$ and $T_z \in S$ are elementary abelian 2-groups, and we will consider them as $\F_2$-vector spaces. Note that $|T|=q^{2t}$, $|H|=2^t$, and $|T_z|=2^t$. Let $x$ and $y$ be the two points of $\mathcal H$ on the line $\ell_\infty$. Then, for $z\in \ell_\infty \setminus \{x,y\}$, it follows that $|T_z\cap H|=2$, since a non-trivial $(z,\ell_\infty)$-elation fixing $\mathcal H$ is necessarily an involution. This suffices to conclude that $H$ intersects each element of $S\setminus \{T_x,T_y\}$ in a one-dimensional subspace, and therefore $H$ is a scattered space with respect to $S$ of dimension $t$ in $V(2t,2)$.
\end{proof}

\subsection{Translation caps in affine spaces}\label{subsec:caps}
The next applications is a generalisation to higher dimensional spaces of the correspondence between translation hyperovals and scattered spaces. This comes from recent work \cite{BaGiMaPoPrep}, to which we refer for the details. Here we only give a sketch of this correspondence.

As mentioned above, the ABB construction is a special case, with $r=2$, of the more general construction of the $2-(q^{rt},q^t,1)$ design $\D(S)$, and when the spread $S$ is Desarguesian, the design $\D(S)$ is an affine space ${\mathrm{AG}}(r-1,q^t)$. Generalising the previous construction, of a translation hyperoval from a scattered space in $\PG(2t-1,2)$, now starting from a scattered space $U$ w.r.t $\D_{r,t,2}$ in $\PG(rt-1,2)$ one obtains a set of points ${\mathcal{K}}=K_U\setminus U$ in the affine space ${\mathrm{AG}}(r-1,2^t)$. Again this set of points satisfies the property that no three of them are collinear. Such a set is called a {\it cap}, and this particular construction gives a {\it translation cap}. Translating this correspondence from \cite{BaGiMaPoPrep} into our terminology gives the following.
\begin{theorem}
A scattered subspace w.r.t. $\D_{r,t,2}$, $t>1$, corresponds to a translation cap in ${\mathrm{AG}}(r-1,2^t)$ and viceversa.
\end{theorem}
This correspondence leads to the existence of complete caps whose cardinality is close to the theoretical lower bound for complete caps. See \cite{BaGiMaPoPrep} for further details.

\subsection{Linear sets}\label{subsec:linear_sets}
Linear sets have many interesting aspects, and as it would take us too much time to elaborate on all of these. 
We refer to \cite{LaVa2015}  and \cite{Polverino2010} for surveys on the topic. Before we explain some of the applications of scattered spaces to the theory of linear sets, we briefly introduce the notion of a linear set using the notation and terminology of field reduction which was formalised in \cite{LaVa2015}.
The technique called "field reduction" is based on the well understood concept of subfields in a finite field, 
and, maybe surprisingly, has proved to be a very powerful tool in Galois Geometry.
Consider the {\em field reduction map} ${\mathcal{F}}_{r,t,q}$ as in \cite{LaVa2015} from $\PG(r-1,q^t)$ to $\PG(rt-1,q)$. Points of $\PG(r-1,q^t)$ are mapped onto $(t-1)$-spaces of $\PG(rt-1,q)$, and in particular the image of the set $\mathcal P$ of points of $\PG(r-1,q^t)$ forms a Desarguesian $(t-1)$-spread $\D_{r,t,q}$ of $\PG(rt-1,q)$. If $U$ is a subspace of $\PG(rt-1,q)$ then by $\B(U)$ we denote the set of points corresponding to the spread elements which have non-trivial intersection with $U$, i.e.
\begin{eqnarray}
\B(U)=\{x \in {\mathcal{P}}~:~{\mathcal{F}}_{r,t,q}(x)\cap U \neq \emptyset\}.
\end{eqnarray}
Here the set $\B(U)$ is considered as a set of points in $\PG(r-1,q^t)$, but using the one-to-one correspondence between $\mathcal P$ and $\D_{r,t,q}$ given by the field reduction map ${\mathcal{F}}_{r,t,q}$, sometimes $\B(U)$ is
also considered as a subset of $\D_{r,t,q}$, consistent with the notation we used in the previous sections. This just means that the sets $\B(U)$ and ${\mathcal{F}}_{r,t,q}(\B(U))$ are sometimes identified. The context (i.e. the ambient space) should always clarify if $\B(U)$ is considered as a subset of $\mathcal P$ or as a subset of $\D_{r,t,q}$.

A set of points $L$ in $\PG(r-1,q^t)$ is called an {\em $\F_q$-linear set} if there exists a subspace $U$ in $\PG(rt-1,q)$ such that $L=\B(U)$. An $\F_q$-linear set $\B(U)$ is said to have {\em rank $m$} if $U$ has projective dimension $m-1$. These definitions immediately lead to the following proposition.
\begin{proposition}
An $\F_q$-linear set $\B(U)$ in $\PG(r-1,q^t)$ has maximal size (w.r.t. its rank) if and only if $U$ is scattered w.r.t. ${\mathcal{D}}_{r,t,q}$.
\end{proposition}
An $\F_q$-linear set $L$ of rank $m$ has at most $(q^m-1)/(q-1)$ points and if this bound is reached then $L$ is called a {\em scattered linear set}. If a $\B(U)$ is an $\F_q$-linear set in $\PG(r-1,q^t)$ and $U$ is maximum (respectively maximally) scattered w.r.t. $\D_{r,t,q}$, then $\B(U)$ is called a {\em maximum} (respectively {\em maximally}) {\em scattered $\F_q$-linear set}. 

In \cite{Polverino2010} Polverino introduced the notion of the dual linear set. If $\beta$ is a non-degenerate sesquilinear form on $\F_{q^t}^r$ and $Tr$ denotes the trace map from $\F_{q^t}$ to $\F_q$, then
$Tr\circ \beta$ defines a non-degenerate form from $\F_{q^t}^r$ to $\F_q$. If $\perp$ denotes the corresponding polarity in $\PG(rt-1,q)$, and $\B(U)$ is an $\F_q$-linear set in $\PG(r-1,q^t)$, then $\B(U^\perp)$ is called the {\em dual linear set with respect to $\beta$}. If $\B(U)$ has rank $m$ then $\B(U^\perp)$ has rank $rt-m$. For maximum
scattered linear sets we have the following theorem.

\begin{theorem}\cite[Theorem 3.5.]{Polverino2010} If $rt$ is even and $\B(U)$ is a maximum scattered $\F_q$-linear set of $\PG(r-1,q^t )$, then the dual linear set with respect to any polarity of $\PG(r-1,q^t)$ is a maximum scattered 
$\F_q$-linear set as well.
\end{theorem}

One of the important questions regarding linear sets is the {\em equivalence problem}. Contrary to linear subspaces (which are  equivalent if and only if they have the same dimension), two linear sets of the same rank are not necessarily equivalent under the action of the projective group or the collineation group of the ambient projective space. This is of course not surprising since two linear sets of the same rank might even have different cardinalities. In few cases the equivalence problem has been solved. 
\begin{theorem}\cite{LaVa2010}{ }\footnote{It was pointed out in \cite{CsZa2015} that one of the conditions of Theorem 3 in \cite{LaVa2010} is not necessary for the equivalence of two linear sets. The condition is however sufficient, and hence does not affect the equivalences stated here.}
All scattered $\F_q$-linear sets of rank $3$ in $\PG(1,q^3)$ and $\PG(1,q^4)$ are equivalent under ${\mathrm{PGL}}(2,q^4)$.
All scattered $\F_q$-linear sets of rank $3$ in $\PG(1,2^5)$ are equivalent under ${\mathrm{P\Gamma L}}(2,2^5)$.\end{theorem}

The equivalence of maximum scattered $\F_q$-linear sets in $\PG(1,q^3)$ can be generalised to all projective spaces  of odd dimension over $\F_{q^3}$. The following was shown for $n=2$ in \cite[Proposition 2.7]{MaPoTr2007} and for general $n$ in \cite[Theorem 4]{LaVa2013}.
\begin{theorem}
All maximum scattered $\F_q$-linear sets in $\PG(2n-1,q^3)$ are ${\mathrm{P\Gamma L}}$-equivalent.
\end{theorem}

This equivalence does however not generalise to maximum scattered $\F_q$-linear sets of odd dimensional projective spaces over extension fields of degree $>3$.
For instance, in \cite{LuMaPoTr2014} it was shown that in $\PG(2n-1,q^t)$, $q>3$, $t\geq4$, there exist inequivalent maximum scattered linear sets.

Another interesting question concerning linear sets is the {\em intersection problem}. Given two linear sets $L_1$ and $L_2$ of given rank in a given projective space, what are the possibilities for the intersection $L_1\cap L_2$? Again the answer is trivial for subspaces, and the question has also been answered for subgeometries (see \cite[Theorem 1.3]{DoDu2008}), but for linear sets the problem is much more complicated and only few results are known. 
The following theorem gives an answer to the intersection problem for an $\F_q$-linear set of rank $k$ and a
scattered $\F_q$-linear set of rank two (i.e. an $\F_q$-subline).
\begin{theorem}\cite[Theorem 8 and 9]{LaVa2010} An $\mathbb{F}_q$-subline intersects an $\F_q$-linear set of rank $k$ of $\PG(1,q^h)$ in $0,1,\ldots,\min\{q+1,k\}$ or $q+1$ points and for every subline $L\cong\PG(1,q)$ of $\PG(1,q^h)$, there is a linear set $S$ of rank $k$, $k\leq h$ and $k\leq q+1$, intersecting $L$ in exactly $j$ points, for all $0\leq j\leq k$.
\end{theorem}
In \cite[Proposition 5.2]{DoDu2014} the authors determined the intersection of two scattered $\F_q$-linear sets of rank $t+1$ in $\PG(2,q^t)$.
Further results on the intersection of (not necessarily scattered) linear sets can be found in \cite{LaVa2013} and \cite{Pepe2011}.

\subsection{Two-intersection sets}
A two-intersection set w.r.t. $k$-dimensional spaces in $\PG(V)$ is a set $\Omega$ of points such that the size of the intersection of the set $\Omega$ with a $k$-space only takes two different values, say $m_1$ and $m_2$. The numbers $m_1$ and $m_2$ are called the {\em intersection numbers} of the set $\Omega$. A fundamental result which makes scattered spaces particularly interesting is the following.

\begin{theorem}{\rm\cite{BlLa2000}}
If $U$ is a scattered space w.r.t. ${\mathcal{D}}_{r,t,q}$ with $\dim U=rt/2$, then $B(U)$ is a two-intersection set w.r.t. hyperplanes in $\PG(r-1,q^t)$, with intersection numbers
$$
m_1=\frac{q^{\frac{rt}{2}-t} -1}{q -1} ~\mbox{and}~
m_2=\frac{q^{\frac{rt}{2}-t+1} -1}{q -1}.
$$
\end{theorem}
If $t$ is even, then this set has the same parameters as the union of 
$(q^{t/2}-1)/(q-1)$ pairwise disjoint Baer subgeometries isomorphic to $\PG(r-1,q^{t/2})$ (call such a set of {\it type I}).
If $t$ is odd, then this set has the same parameters as the union of 
$(q^t-1)/(q-1)$ elements of an $(r/2-1)$-spread in $\PG(r-1,q^t)$. 
We call these two-intersection sets of {\it type II}.
It was proved for $r=3$ and $t=4$ in \cite{BaBlLa2000} and for general $rt$ even in \cite{BlLa2002} that the two-intersection sets obtained from scattered spaces are not of these types.
\begin{theorem}{\rm\cite{BlLa2002}}\label{thm:nonequiv}
A scattered $\F_q$-linear set of rank $rt/2$ in $\PG(r-1,q^t)$ is 
inequivalent to the two-intersection sets of type I or type II.
\end{theorem}

\subsection{Two-weight codes}
An $\F_q$-linear $[n,k]$-code $C$ is a $k$-dimensional subspace of $\F_q^n$. Vectors belonging to $C$ are called {\em codewords} and the {\em weight} ${\mathrm{wt}}(c)$ of a codeword is the number of nonzero coordinates of $c$ with respect to some fixed basis of $\F_q^n$. The {\it distance} $d(c_1,c_2)$ between two codewords is the number of positions in which they have different coordinates and is thus equal to $\wt(c_1-c_2)$. If the nonzero minimum distance of $C$ is $d$, then the code is called an {\em $\F_q$-linear $[n,k,d]$-code}. In this case the code $C$ is an $e$-error correcting code with $e=\lfloor \frac{1}{2}(d-1)\rfloor$. 

Given a two-intersection set $\B(U)$ from a maximum scattered space $U$ in $\PG(rt-1,q)$, we can obtain a two-weight code as follows. We briefly sketch the construction and refer to Calderbank et al. \cite{CaKa1986} for further details. Put $n=|\B(U)|$ and define the code $C_U$ as the subspace generated by the columns of the $(n\times r)$-matrix $M_U$ whose rows are the coordinates of the points of $\B(U)$ with respect to some fixed frame of $\PG(r-1,q^t)$. Then $C_U$ has length $n$ and dimension $r$ (for this we use that $U$ is maximum scattered). Since $\B(U)$ has intersection numbers $m_1$ and $m_2$, the code $C_U$ is a two-weight code with weights $n-m_1$ and $n-m_2$. Hence we have the following theorem.
\begin{theorem}{\rm \cite{BlLa2000}}
If $U$ is a scattered space of dimension $m=rt/2$ w.r.t. ${\mathcal{D}}_{r,t,q}$, then $C_U$ is an $\F_q$-linear
$[(q^m-1)(q-1),r]$-code with weights
$$
q^{m-t}\left ( \frac{q^t-1}{q-1}\right ) ~~\mbox{and}~~q^{m-t+1}\left ( \frac{q^{t-1}-1}{q-1}\right ).
$$
\end{theorem}

\subsection{Blocking sets}\label{subsec:blocking_sets}
Blocking sets have received a tremendous amount of attention in the past decades, and it is through
research in this area that scattered spaces came into the spotlight.
In particular, the results by Blokhuis et al. from \cite{BlStSz1999} where it was shown that
an $s$-fold blocking set in $\PG(2,q^4)$ of size
$s(q^4+1)+c$, with $s$ and $c$ small enough, contains the union of $s$ disjoint Baer subplanes, motivated
the paper by Ball et al. \cite{BaBlLa2000} from 2000. In the latter paper the authors constructed
a scattered linear set of rank 6, thus obtaining a $(q+1)$-fold
blocking set of size $(q+1)(q^4+q^2+1)$ in $\PG(2,q^4)$, and they proved that it is not the union of Baer subplanes
(see also Theorem \ref{thm:nonequiv}).

A more general result on scattered spaces and blocking sets is the following theorem from \cite{BlLa2000}. This shows that also scattered spaces which are not maximum generate blocking sets.
\begin{theorem}{\rm \cite{BlLa2000}}\\
A scattered subspace $U$ of dimension $m$, with respect to a
Desarguesian $t$-spread, in ${\mathrm V}(rt,q)$ induces a $\left (
\theta_{k-1}(q)  \right )$-fold blocking set $\B(U)$, with respect to
$(\frac{rt-m+k}{t}-1)$-dimensional subspaces in $\mathrm{PG}(r-1,q^t)$,
of size $\theta_{m-1}(q) $, where $1\leq k \leq m$ such that
$t~|~(m-k)$.
\end{theorem}

\subsection{Embeddings of Segre varieties}
The {\em Segre variety $S_{t,t}(q)$} is an algebraic variety in $\PG(t^2-1,q)\cong \PG(\F_q^t\otimes \F_q^t)$, whose points correspond to the fundamental tensors in $\F_q^t\otimes \F_q^t$.
The geometry of points and lines lying on the Segre variety $S_{t,t}(q)$ is a {\em semilinear space} (also called
a {\em product space} representing the product $\PG(t-1,q)\times \PG(t-1,q)$).
A \textit{projective embedding} of a semilinear space is an injective map into a projective space
mapping lines into lines. So, $S_{t,t}(q)$ is a projective embedding of the product space
$\PG(t-1,q)\times \PG(t-1,q)$ in $\PG(t^2-1,q)$.
By \cite{Zanella1996}, any embedded product space is an injective projection of a Segre variety.
Since the image of any embedding of $\PG(t-1,q)\times \PG(t-1,q)$ into a projective space $\PG(m,F)$ 
contains two disjoint $(t-1)$-subspaces, it holds that $m\ge 2t-1$. 
Therefore, a projective embedding of
$\PG(t-1,q)\times \PG(t-1,q)$ into $\PG(2t-1,F)$ is called a  \textit{minimum embedding}.
In \cite{LaShZa2015} a construction is given of such a minimum embedding using a maximum scattered
 subspace w.r.t. a Desarguesian spread.
\begin{theorem}{\rm \cite{LaShZa2015}}
If $U$ is a maximum scattered subspace w.r.t. ${\mathcal{D}}_{2,t,q}$, then
$\B(U)\subset \PG(2t-1,q)$ is a minimum embedding of the Segre variety $S_{t,t}(q)$.
\end{theorem}
The smallest non-trivial example of a Segre variety $S_{t,t}(q)$ is the hyperbolic quadric $\PG(3,q)$ for $t=2$. 
In \cite{LaShZa2015}, it is shown that there exists an embedding $\B(U)$ of $S_{t,t}(q)$ which is also a 
 hypersurface of degree $t$ in $\PG(2t-1,q)$, extending the properties of the hyperbolic quadric in $\PG(3,q)$.
By construction this embedding is covered by two systems of maximum subspaces (in this case $(t-1)$-dimensional). However, unlike the Segre variety, it turns out that
the embedding $\B(U)$ contains $t$ systems of maximum subspaces, and hence for $t>2$, contrary
to what one might expect, there exist systems of maximum subspaces which are not the image of maximum subspaces of the Segre variety, see \cite[Theorem 6]{LaShZa2015}.

\subsection{Pseudoreguli}

The concept of a pseudoregulus is a generalisation of the concept of a regulus. If $A$, $B$, $C$ are three
distinct $(n-1)$-dimensional subspaces contained in a common $(2n-1)$-space, then through each point
of any of these three subspaces there is exactly one line intersecting each of the spaces $A$, $B$ and $C$.
Such a line is called a {\it transversal line} w.r.t. $A$, $B$, and $C$. If each transversal line $\ell$ 
is given coordinates with respect to the frame $A\cap \ell$, $B\cap \ell$, and $C\cap \ell$, then the points on 
all the transversals with the same coordinates form an $(n-1)$-dimensional subspace. 
The set of these subspaces is called the
{\em regulus $R(A,B,C)$ determined by $A$, $B$ and $C$}, and the transversal lines w.r.t. $A$, $B$ and $C$ are also called {\it transversal lines of the regulus $R(A,B,C)$}.  A regulus consisting of $(n-1)$-dimensional subspaces is also called an {\em $(n-1)$-regulus}.
Equivalently, the regulus $R(A,B,C)$ is the family of maximal subspaces containing $A$, $B$, and $C$ of the 
unique Segre variety $S_{2,n}(q)$ containing $A$, $B$ and $C$. The transversal lines form the other family of maximal subspaces of $S_{2,n}(q)$.
If $n=2$ these become the two families of $q+1$ lines lying on a hyperbolic quadric in $\PG(3,q)$.
In 1980, Freeman constructed a set of $q^2+1$ lines in $\PG(3,q^2)$ which have exactly 2 transversal lines, called
a {\em pseudoregulus}. The $q^2+1$ lines are the extended lines of a Desarguesian spread in a 
Baer subgeometry $\PG(3,q)$, and the transversal lines are the two conjugate lines defining the spread.  
This idea was extended by Marino et al. in \cite{MaPoTr2007}, to a set of $q^3+1$ lines in $\PG(3,q^3)$ using a maximum scattered $\F_q$-linear set (of rank 6), obtaining the following.
\begin{proposition}{\rm\cite{MaPoTr2007}}
To any scattered $\F_q$-linear set $L$ of rank 6 of $\PG(3, q^3)$ is associated an $\F_q$-pseudoregulus $L$ consisting of all $(q^2 + q + 1)$-secant lines of $L$.
\end{proposition}
Instead of a Baer subgeometry in $\PG(3,q^2)$ the authors of \cite{MaPoTr2007} considered a subgeometry $\PG(5,q)$ in $\PG(5,q^3)$ and a Desarguesian plane spread $\D_{4,3,q}$ of $\PG(5,q)$ together with the three conjugate lines $\ell$, $\ell^\omega$ and $\ell^{\omega^2}$ defining the spread $\D_{4,3,q}$. Projecting the subgeometry from $\ell$ onto the 3-dimensional space $\langle \ell^\omega,\ell^{\omega^2}\rangle$ gives a scattered linear set of rank 6. The transversals to the associated $\F_q$-pseudoregulus are the lines $\ell^\omega$ and $\ell^{\omega^2}$.
\begin{remark}
Note that in $\PG(3,q^3)$ the pseudoregulus can thus be reconstructed from the scattered linear set of rank 6. Simply take all the $(q^2+q+1)$-secants. This was not the case in the previous situation in $\PG(3,q^2)$, where the pseudoregulus cannot be reconstructed from the Baer subgeometry $\PG(3,q)$ (or equivalently a maximum scattered linear set of rank 4). In this case any Desarguesian spread of the Baer subgeometry gives a pseudoregulus.
\end{remark}
These ideas were further developed in \cite{LaVa2013} in higher dimensional projective spaces, and the following theorem was proved.
\begin{theorem}{\rm \cite{LaVa2013}} If $L$ is a scattered $\F_q$-linear set of rank $3n$ in $\PG(2n-1,q^3)$, $n\geq 2$, then a line of $\PG(2n-1,q^3)$ meets $L$ in $0,1,q+1$ or $q^2+q+1$ points and every point of $L$ lies on exactly one $(q^2+q+1)$-secant to $L$. Two different $(q^2+q+1)$-secants to $L$ are disjoint and there exist exactly two $(n-1)$-spaces, meeting each of the $(q^2+q+1)$-secants in a point.
\end{theorem}
So in this case, the transversals are no longer lines, but $(n-1)$-dimensional subspaces, and again the pseudoregulus is uniquely determined by the maximum scattered linear set.
This leads to the following questions. Does every pseudoregulus (uniquely) determine a scattered subspace? And
how do we recognise a pseudoregulus, given a set of mutually disjoint lines? The following theorem gives a geometric characterisation of a regulus and pseudoregulus in $\PG(3,q^3)$, giving a partial answer to the second question.
In the following results, $\tilde{{\mathcal{L}}}$ denotes the set of points contained in the lines of the set $\mathcal L$.
\begin{theorem} {\rm \cite[Theorem 24]{LaVa2013}}
Let ${\mathcal{L}}$ be a set of $q^3+1$ mutually disjoint lines in $\PG(3,q^3)$, $q>2$. If each subline defined by three collinear points of $\tilde{{\mathcal{L}}}$ is contained in $\tilde{{\mathcal{L}}}$, then ${{\mathcal{L}}}$ is a regulus or a pseudoregulus.
\end{theorem}
Also the first question was answered in \cite{LaVa2013}: it is possible to reconstruct the maximum scattered $\F_q$-linear set from an $\F_q$-pseudoregulus in $\PG(2n-1,q^3)$, but this scattered linear set is not unique. 
\begin{theorem}{\rm \cite{LaVa2013}}
Let $q > 2$, $n\geq 2$. Let ${\mathcal{L}}$ be a pseudoregulus in $\PG(2n-1, q^3)$, let $P$ be a point of $\tilde{{\mathcal{L}}}$, on the line $\ell$ of ${{\mathcal{L}}}$, not lying on one of the transversal spaces to ${{\mathcal{L}}}$. Let $T = \{\ell_1,\ell_2,\ldots\}$ be the set of $(q + 1)$-secants through $P$ to $\tilde{{\mathcal{L}}}$, let $P(T)$ be the set of points on the lines of $T$ in $\tilde{{\mathcal{L}}}$. Let $\pi_i$ be the plane $\langle \ell,\ell_i\rangle$, and let $D_i$ be the set of directions on $\ell$, determined by the intersection of $\pi_i$ with $\tilde{{\mathcal{L}}}$. Then $D_i = D_1$, for all $i$, and $P(T)$, together with the points of $D_1$, form a scattered $\F_q$-linear set of rank $3n$ determining the pseudoregulus ${{\mathcal{L}}}$.
\end{theorem}
It follows from the proof of the above theorem in \cite{LaVa2013}, that the scattered linear set is not uniquely determined by the pseudoregulus. In fact we have the following.
\begin{corollary}
Let $q > 2$. If ${{\mathcal{L}}}$ is a pseudoregulus in $\PG(2n-1, q^3)$, then there are $q-1$ scattered $\F_q$-linear sets having ${{\mathcal{L}}}$ as associated pseudoregulus.
\end{corollary}

The property that any maximum scattered $\F_q$-linear set in $\PG(2n-1,q^t)$ gives rise to a pseudoregulus no longer holds for $t>3$. This was further investigated in \cite{LuMaPoTr2014}, introducing {\em maximum scattered linear sets of pseudoregulus type}. 
Also, all maximum scattered $\F_q$-linear sets in $\PG(2n-1,q^t)$ are no longer (projectively) equivalent for $t>3$. Indeed the following was proved in \cite{LuMaPoTr2014}.
\begin{theorem}{\rm \cite{LuMaPoTr2014}} For $n\geq 2$ and $t>3$, (i) there exist maximum scattered $\F_q$-linear sets which are not of pseudoregulus type, and (ii) there are $\varphi(t)/2$ orbits of maximum scattered $\F_q$-linear sets of $\PG(2n-1,q^t)$ of  pseudoregulus type under the action of the collineation group ${\mathrm{P\Gamma L}}(2n,q^t)$.
\end{theorem}
In the statement of this theorem $\varphi(t)$ denotes Euler's totient function, i.e. the number of integers smaller than $t$ and relatively prime to $t$. Also the maximum scattered $\F_q$-linear sets in the theorem are always of rank $nt$,  and all maximum scattered $\F_q$-linear sets of pseudoregulus type are of the form $L_{\rho,f}$, see \cite{LuMaPoTr2014}.

The authors of \cite{LuMaPoTr2014} also introduced maximum scattered $\F_q$-linear sets of pseudoregulus type on a projective line, and these were further investigated in \cite{CsZaPrep}. Of course there is no longer a pseudoregulus associated to such a linear set, but the definition is stimulated from the algebraic formula for maximum scattered $\F_q$-linear sets $L_{\rho,f}$ of pseudoregulus type in higher dimensions. 

\subsection{Semifield theory}\label{subsec:semifields}
The term "semifield" is a short term introduced by Knuth in 1965 \cite{Knuth1965} for a non-associative division algebra. The algebraic structures themselves were first studied by Dickson in 1906 \cite{Dickson1906}. We restrict ourselves to the finite case.
A {\it finite semifield} $\mathbb S$ is an algebra with at least two elements, and two binary operations $+$ and $\circ$, satisfying the following axioms.
\begin{itemize}
\item[(S1)] $({\mathbb{S}},+)$ is a group with identity element $0$.
\item[(S2)] $x\circ(y+z) =x\circ y + x\circ z$ and $(x+y)\circ z = x\circ z + y
\circ z$, for all $x,y,z \in {\mathbb{S}}$.
\item[(S3)] $x\circ y =0$ implies $x=0$ or $y=0$.
\item[(S4)] $\exists 1 \in {\mathbb{S}}$ such that $1\circ x = x \circ 1 = x$, for all $x \in {\mathbb{S}}$.
\end{itemize}
An algebra satisfying all of the axioms of a semifield except (S4) is called a {\it pre-semifield}.
An important example of a finite pre-semifield is the Generalized Twisted Field (GTF) due to Albert \cite{Albert1961}, obtained by defining the multiplication $x \circ y=xy-cx^\alpha y^\beta$ on the finite 
field $\F_{q^n}$, where $\alpha, \beta \in {\mathrm{Aut}}(\F_{q^n})$ with $Fix(\alpha)=Fix(\beta)=\F_q$, and $c\in \F_{q^n}$ with $N_{\F_{q^n}/ \F_q}(c)\neq 1$.

Finite semifields have been studied intensively by a variety of mathematicians and they have many connections with structures in Galois Geometry. We refer to \cite{Knuth1965}, \cite{Kantor2006}, \cite{Lavrauw2013} and to the chapter \cite{LaPo2011} and its references for the necessary background, an overview of the results, connections, and further reading. 
A semifield $\SSS$ defines a projective plane $\pi(\SSS)$, called a {\em semifield plane},
and isomorphism classes of semifield planes correspond to {\em isotopism classes} of semifields, a result from Albert \cite{Albert1960}.
We restrict ourselves here to what we call scattered semifields, which  
we define as follows. 
Suppose $\SSS$ is a semifield which is $l$-dimensional over its left nucleus and $ls$-dimensional over its center $\F_q$.
Consider the set $R(\SSS)$ of endomorphisms of $\F_{q^s}^l$ corresponding to right multiplication in the semifield $\SSS$. This set is called a {\em semifield spread set}, and $R(\SSS)$ is an $\F_q$-linear set of rank $ls$ in the projective space $\PG(l^2-1,q^s)$. Now if $R(\SSS)$ is a scattered $\F_q$-linear set, then $\SSS$ is called a {\em scattered semifield}.
It is known that $R(\SSS)$ is disjoint from the $(l-2)$--th secant variety $\Omega(S_{l,l}(q^s))$ of
a Segre variety $S_{l,l}(q^s)$ and this connection also gives a nice interpretation of the isotopism classes.
\footnote{A relation between the isotopism classes of semifields and equivalence classes of {\em embeddings} of Segre varieties can be found in \cite{LaZa2014}.}
\begin{theorem}[from \cite{Lavrauw2011}]\label{thm:spreadequivalent}
The isotopism class of a semifield $\SSS$ corresponds to the orbit of $R(\SSS)$ under the action
of the group ${\mathcal{H}}\leq {\mathrm{PGL}}(l^2,q^s)$ preserving the two systems of maximal
subspaces contained in the Segre variety $S_{l,l}(q^s)$ in $\PG(l^2-1,q^s)$.
\end{theorem}
It follows from Theorem \ref{thm:spreadequivalent} that being "scattered" is an isotopism invariant, and this
makes it a useful tool to investigate the isotopism problem, which is usually a very hard problem: given two semifields
$\SSS_1$ and $\SSS_2$, decide whether they are isotopic or not.

The power of this geometric approach is illustrated in
\cite{CaPoTr2006} ($l=s=2$), and \cite{MaPoTr2007} ($l=2$, $s=3$), and
by the recent results from \cite{LaMaPoTr2015a} and \cite{LaMaPoTr2015b}, where the structure of $R(\SSS)$ is used to solve the isotopism problem regarding the semifields constructed by Dempwolff in \cite{Dempwolff2013}.

\subsection{Splashes of subgeometries}

Given a subgeometry $\pi_0$ and a line $l_\infty$ in a projective space $\pi$, by extending the hyperplanes of $\pi_0$ to hyperplanes of $\pi$ and intersecting these with the line $l_\infty$, one obtains a set of points on the projective line $l_\infty$. Precisely, if we denote the set of hyperplanes of a projective space $\pi$ by ${\mathcal{H}}(\pi)$, and $\overline{U}$ denotes the extension of a subspace $U$ of the subgeometry $\pi_0$ to a subspace of $\pi$,  then we obtain the set of points $\{ l_\infty \cap \overline{H} ~:~H \in {\mathcal{H}}(\pi_0)\}$.
These sets have been studied in \cite{LaZa2015} generalising the initial studies in \cite{BaJa15} where the {\it splash of $\pi_0$ on $l_\infty$} was introduced for Desarguesian planes and cubic extensions, i.e. for a subplane $\pi_0\cong\PG(2,q)$ in $\pi\cong\PG(2,q^3)$. If $l_\infty$ is tangent to (respectively, disjoint from) $\pi_0$, then a splash is called the {\it tangent splash} (respectively, {\it external} or {\em exterior splash}) 
{of $\pi_0$ on $l_\infty$}. Note that when $l_\infty$ is secant to $\pi_0$, the splash of $\pi_0$ on $l_\infty$ is just a subline.

One of the main results of \cite{LaZa2015} shows the equivalence between splashes and linear sets on a projective line.
\begin{theorem}
Let $r,n>1$.
If $S=S(\pi_0,l_\infty)$ is the splash of the $q$-subgeometry $\pi_0$ of $\PG(r-1,q^n)$ on the line $l_\infty$, 
then $S$ is an $\F_q$-linear set 
of rank $r$. Conversely, if $S$ is an ${\mathbb{F}}_q$-linear set of rank $r$ on the line $l_\infty\cong\PG(1,q^n)$, then there exists an embedding of
$l_\infty$ in $\PG(r-1,q^n)$ and a $q$-subgeometry $\pi_0$ of $\PG(r-1,q^n)$ such that $S=S(\pi_0,l_\infty)$.
\end{theorem}
Also, it is shown that the number of hyperplanes through a point determines the weight of that point in the linear set. This leads to the following characterisation of scattered linear sets.
\begin{theorem}
Let $S$ be the splash of a subgeometry $\pi_0\cong\PG(r-1,q)$ of $\pi\cong\PG(r-1,q^n)$ on $l_\infty\cong\PG(1,q^n)$. Then $S$ is a scattered linear set if and only if $S$ is an external splash, where every point of $S$ is on exactly one hyperplane of $\pi_0$.
\end{theorem}

\subsection{MRD-codes}
A {\em rank metric code} is a set of $(m\times n)$-matrices over some field where the distance between two codewords is defined as  the matrix rank of their difference. A {\em maximum rank distance code} ({\em MRD} code) is a rank metric code of maximum size with respect to its distance. Precisely, if $C\subset M_{m\times n}(\F_q)$, with minimum distance $d$, then $C$ is an MRD code if and only if $|C|=q^{nk}$ with $k=m-d+1$.
Such codes were constructed by Delsarte \cite{Delsarte1978} and Gabidulin \cite{Gabidulin1985} for every $m$, $n$, and $d$. 

Recently, Sheekey constructed new families of MRD codes in \cite{Sheekey2015}. They generalise the previously known Gabidulin codes. 
These codes are called {\em twisted Gabidulin codes} because of their analogy with the generalized twisted fields constructed by Albert (see Section \ref{subsec:semifields}).

In the same paper \cite{Sheekey2015}, the author gives an interesting connection between maximum scattered linear sets on a projective line $\PG(1,q^t)$, and MRD-codes of dimension $2t$ over $\F_q$ and minimum distance $t-1$.
If $U$ is a $t$-dimensional subspace of $\F_q^{2t}$, then 
$U$ can be represented by $\{(x,f(x))~:~x \in \F_{q^t}\}$ for
some $\F_q$-linear polynomial $f(Y)\in \F_q[Y]$. Now $U$ is scattered
w.r.t. the Desarguesian spread $D_{2,t,q}$ if for each $a\in \F_{q^t}^*$, the dimension of the intersection of $U$ with
$D_a=\{(x,ax)~:~x \in \F_{q^t}\}$ has dimension at most 1. This is equivalent with the condition that the rank of the $\F_q$-linear map $x\mapsto ax+bf(x)$ is at least $t-1$, for all $a,b \in \F_{q^t}$, $(a,b)\neq (0,0)$. 
Then $C_{U}=\{\varphi(ax+bf(x))~:~a,b \in \F_{q^t}, (a,b)\neq (0,0)\}$ is an $\F_q$-subspace of the vector space of $(t\times t)$-matrices, where $\varphi$ is an isomorphism associating a $(t\times t)$-matrix $M_g$ to each linearized polynomial $g(x)$ w.r.t. some fixed basis.
\begin{theorem}{\rm\cite{Sheekey2015}}
If $U$ is a maximum scattered space w.r.t. $\D_{2,t,q}$, then $C_{U}$
is an $\F_q$-linear MRD-code of dimension $2t$ and minimum distance $t-1$, and conversely.
\end{theorem}
This one-to-one correspondence is particularly nice since it preserves equivalence.
\begin{theorem}{\rm\cite{Sheekey2015}}
Two scattered $\F_q$-linear sets $\B(U)$ and $\B(U' )$ of rank $t$ are equivalent in $\PG(1,q^t)$ if and only if $C_U$ and $C_{U'}$ are equivalent MRD-codes.
\end{theorem}
We refer to \cite{Sheekey2015} for examples and further details.

\paragraph{Acknowledgment}
The author thanks the members of the Algebra group of Sabanc\i{} University for their hospitality and the anonymous referees for their comments and suggestions.


\end{document}